\documentclass[reqno]{amsart}

\usepackage{amsmath}
\usepackage{amsfonts}
\usepackage{amssymb,enumerate}
\usepackage{amsthm}
\usepackage[all]{xy}
\usepackage{rotating}
\usepackage{hyperref}
\usepackage{color}

\theoremstyle{plain}
\newtheorem{lem}{Lemma}[section]

\newtheorem{prop}[lem]{Proposition}
\newtheorem{thm}[lem]{Theorem}

\newtheorem*{mthm*}{Main Theorem}

\theoremstyle{definition}

\newtheorem{ex}[lem]{Example}
\newtheorem{question}[lem]{Question}

\newtheorem{para}[lem]{}

\newtheorem*{convention*}{Convention}

\newcommand{\im}{\operatorname{Im}}

\newcommand{\Tor}{\operatorname{Tor}}

\newcommand{\HH}{\operatorname{H}}

\newcommand{\Ker}{\operatorname{Ker}}

\newcommand{\shift}{\mathsf{\Sigma}}

\newcommand{\xra}{\xrightarrow}
\newcommand{\xla}{\xleftarrow}

\renewcommand{\geq}{\geqslant}
\renewcommand{\leq}{\leqslant}
\renewcommand{\ker}{\Ker}

\def\Tor{\operatorname{Tor}}

\def\m{\mathfrak{m}}
\def\fm{\mathfrak{m}}

\newcommand{\Cone}{\operatorname{Cone}}

\numberwithin{equation}{lem}

\begin{document}

\bibliographystyle{amsplain}

\title[Tor algebra of local rings with decomposable maximal ideal]{Tor algebra of local rings with\\ decomposable maximal ideal}

\author{Saeed Nasseh}
\address{Department of Mathematical Sciences\\
Georgia Southern University\\
Statesboro, GA 30460, U.S.A.}
\email{snasseh@georgiasouthern.edu}

\author{Maiko Ono}
\address{Department of Mathematics, Okayama University, 3-1-1 Tsushima-naka, Kita-ku, Okayama 700-8530, Japan}
\email{onomaiko.math@okayama-u.ac.jp}

\author{Yuji Yoshino}
\address{Graduate School of Environmental, Life, Natural Science and Technology, Okayama University, Okayama 700-8530, Japan}
\email{yoshino@math.okayama-u.ac.jp}

\thanks{M. Ono was partly supported by the Wesco Scientific Promotion Foundation and Y. Yoshino was supported by JSPS Kakenhi Grant 19K03448 and 24K0669.}



\begin{abstract}
Let $(R,\m_R)$ be a commutative noetherian local ring. Assuming that $\m_R=I\oplus J$ is a direct sum decomposition, where $I$ and $J$ are non-zero ideals of $R$, we describe the structure of the Tor algebra of $R$ in terms of the Tor algebras of the rings $R/I$ and $R/J$.
\end{abstract}

\maketitle


\section{Introduction}\label{sec20231210a}

\begin{para}\label{para20250420b}
Throughout the paper, $(R,\fm_R,k)$ is a commutative noetherian local ring, $\mathbf{x}$ is a minimal generating sequence of length $e$ for $\fm_R$, and $\widehat{R}\cong Q/\frak a$ is a minimal Cohen presentation, i.e., $(Q,\frak m_Q, k)$ is a complete regular local ring of Krull dimension $e$ and $\frak a$ is an ideal of $Q$ such that $\frak a\subseteq \frak m_Q^2$. We denote the \emph{Tor algebra} $\Tor^Q(k, \widehat{R})$ by $A_R$, which has a finite-dimensional DG $k$-algebra structure with zero differential. Here, ``DG algebra'' stands for strongly commutative differential graded algebra that is non-negatively graded; for unspecified definitions and notations from DG homological algebra we refer the reader to~\cite{NOY-jadid}, \cite{NOY1}, or~\cite{nasseh:survey}. Also, $K^R(-)$ denotes the Koszul complex on a sequence of elements in $R$ and for an $R$-module $M$, we denote the tensor product $K^R(-)\otimes_RM$ by $K^R(-, M)$.
\end{para}

Several classes of local rings can be characterized by analyzing the structure of the Tor algebra. For instance, by a result of Serre~\cite{serre}, $R$ is regular if and only if $A_R=(A_R)_0$; Assmus~\cite{assmus} and Tate~\cite{Tate} showed that $R$ is a complete intersection if and only if $A_R$ is the exterior algebra on $(A_R)_1$; and Avramov and Golod~\cite{AG} proved that $R$ is Gorenstein if and only if $A_R$ is a Poicaré duality algebra. Other works on the structure of Tor algebra include, but certainly not limited to,~\cite{AKM, buchs, LOJ, kustin:gacfct, MR1292771, kustin:asmrgrecf, MR713381, KellerV, KellerV1, wey}. Deep connections between $R$ and $A_R$ can be explored using a method invented by Avramov, as we discuss next. Kustin~\cite{kustin} calls this method the \emph{Avramov's machine}.

\begin{para}\label{para20240713a}
Let $\mathbf{y}$ be a minimal generating sequence for $\fm_Q$. Note that $K^Q(\mathbf{y})\xra{\simeq} k$ is a minimal free resolution
of $k$ over $Q$ and thus, we obtain a diagram
\begin{equation}\label{eq20240713a}
K^{R}(\mathbf{x})\xra{\simeq}K^{\widehat{R}}(\mathbf{x}\widehat{R})\xla{\cong}K^Q(\mathbf{y})\otimes_Q\widehat{R}\xra\simeq k\otimes_Q \widehat{R}
\end{equation}
of (quasi-)isomorphisms. Therefore, there is an isomorphism
\begin{equation}\label{eq20250502a}
A_R\cong \HH(K^{R}(\mathbf{x}))
\end{equation}
of DG $k$-algebras under which we can identify $A_R$ by $\HH(K^{R}(\mathbf{x}))$.

If we assume that the minimal free resolution $F_{\widehat{R}}^Q$ of $\widehat{R}$ over $Q$ admits a DG algebra structure, then $k\otimes_QF_{\widehat{R}}^Q$ is a DG $k$-algebra with zero differential and in this case we have $A_R\cong k\otimes_QF_{\widehat{R}}^Q$; see~\cite[2.7 and 2.8]{nasseh:survey} for more details.
\end{para}

Using Avramov's machine, especially in case where $F_{\widehat{R}}^Q$ admits a DG algebra structure, many interesting questions about $R$ can be dealt with by translating them into questions about $A_R$. The point is that, although $A_R$ is graded commutative (as oppose to just commutative), it has a more comprehensible structure than $R$. For instance, the Poincaré series of $R$ was described in terms of the Poincaré series of $A_R$ by Avramov; see~\cite{avramov:cslrec3} or~\cite{MR485906}. Also, a method for constructing surjective Golod maps to $R$ using Golod maps to $A_R$ was introduced by Avramov and Backelin~\cite{Av:vpd'}. Moreover, using the Avramov's machine, Nasseh and Sather-Wagstaff~\cite{nasseh:lrfsdc} answered a conjecture posed by Vasconcelos about semidualizing $R$-modules. Further applications of the Tor algebra on the Ext- and Tor-friendliness can be found in~\cite{AINSW, avramov:phcnr}; for a collection of numerous other applications see~\cite{nasseh:survey}.\vspace{4pt}

Assuming that $\m_R=I\oplus J$ is a non-trivial direct sum decomposition, where $I$ and $J$ are ideals of $R$, the purpose of this paper is to describe the structure of the Tor algebra $A_R$ in terms of the Tor algebras $A_{R/I}$ and $A_{R/J}$. We will state the main result of this paper after clarifying some more notations in the next discussion.

\begin{para}\label{para20250420c}
Let $A$ be a DG algebra. By $A^+$ we denote the positive graded part of $A$, i.e., $A^+=\oplus_{n\geq 1}A_n$.
For a DG $A$-module $L$, the trivial extension $A\ltimes L$ is the DG algebra with the underlying complex structure $A\oplus L$ and the product given by the formula $(a,l)(a',l')=(aa',al'+(-1)^{|l||a'|}a'l)$, for homogeneous elements $a,a'\in A$ and $l,l'\in L$. Also, for an integer $i$, the $i$-th shift of $L$ is denoted $\shift^i L$. Note that for all integers $j$ we have $\left(\shift^i L\right)_j = L_{j-i}$ and $\partial_j^{\shift^i L}=(-1)^i\partial_{j-i}^L$, where $\partial$ denotes the differential. We simply write $\shift L$ for $\shift^1 L$. Finally, for a positive integer $n$, we denote by $L^n$ the direct sum $\oplus_{i=1}^n L$.
\end{para}

Here is the statement of the main result of this paper. 

\begin{mthm*}\label{2products}
Assume that $\m_R=I\oplus J$, where $I$ and $J$ are non-zero ideals of $R$ minimally generated by sequences $\mathbf{x}_1$ of length $t$ and $\mathbf{x}_2$ of length $s$, respectively. Then, there is an isomorphism
\begin{equation}\label{eqa}
A _R ^+ \cong \left(\left(\bigwedge k ^{t}\otimes_k A ^+_{R/I}\right)\times \left(A ^+_{R/J} \otimes _k \bigwedge k^{s}\right)\right) \ltimes W
\end{equation}
of $k$-algebras (without unity; see~\ref{para20240724b}) with
\begin{equation}\label{eq20250510a}
W=\shift^{-1}\left(\frac{\bigwedge k^{t}\otimes_k\bigwedge k^{s}}{k\otimes_k\bigwedge k^{s}+\bigwedge k^{t} \otimes_k k}\right)
\end{equation}
which determines the algebra structure of $A_R$ via the equality  $A_R = k \oplus A_R ^+$.
\end{mthm*}


\section{Proof of Main Theorem}\label{section20240729a}

We give the proof of Main Theorem after some preparation. For the time being, to avoid confusion, only the terminology from~\ref{para20250420b}-\ref{para20250420c} (and not from Main Theorem) will be in effect until further notice. We start with the following discussion.

\begin{para}\label{para20240724a}
The exact sequence $0 \to K^R(\mathbf{x}, \m_R)  \to K^R(\mathbf{x})  \xra{\nu}  K^R(\mathbf{x}, k) \to 0$ of DG $K^R(\mathbf{x})$-modules induces a triangle 
\begin{equation}\label{tri}
\HH(K^R(\mathbf{x}, \m_R))  \to A_R \xra{\pi_R} \bigwedge k^{e}  \xra{\psi} \shift \HH(K^R(\mathbf{x}, \m_R))
\end{equation}
of graded $k$-vector spaces. Note that the induced map $\pi_R:=\HH(\nu)$ is a graded $k$-algebra homomorphism. In this situation, we have an equality
\begin{equation}\label{eq20250421a}
Z(K^R(\mathbf{x}, \m_R)) ^+  = Z(K^R(\mathbf{x}))^+
\end{equation}
of positive graded parts of the cycle sets. To justify this, let $\alpha = \sum_{\Lambda} c_{\Lambda} \omega_{\Lambda}$ be an element of degree $0<r\leq e$ in $Z(K^R(\mathbf{x}))^+$ with $c_{\Lambda} \in R$ (not necessarily non-zero) and $\omega_{\Lambda} = \omega_{\lambda_1} \wedge \cdots \wedge \omega_{\lambda_r}$, where $\Lambda = \{\lambda_1, \ldots , \lambda_r\}$ runs through all subsets of $\{1, \ldots , e\}$ of cardinality $r$. For each subset $\Delta\subseteq \{1, \ldots , e\}$ of cardinality $r-1$, looking at the coefficient of $\omega_{\Delta}$ in $d_r^{K^R(\mathbf{x})}(\alpha)= 0$ we get the equality 
$$
\sum _{\lambda_t \not\in \Delta} \pm \left(c_{\Delta\cup \{\lambda_t\}} x_{\lambda_t}\right) =0.
$$ 
Thus, since $\mathbf{x}$ is a minimal generating sequence for $\m_R$, none of the elements $c_{\Lambda}$ is a unit in $R$. 
In other words, all of the elements $c_{\Lambda}$ are in $\m_R$, and hence, $\alpha \in Z\left(K^R(\mathbf{x}, \m_R)\right)^+$. Therefore, the equality~\eqref{eq20250421a} holds.
\end{para}

The following statement is an immediate consequence of the equality~\eqref{eq20250421a}. 

\begin{prop}\label{cor20250422a}
The following equalities hold:
\begin{gather}
\ker\left(\pi_R\right)=A_R^+\label{eq20250422a}\\
\im\left(\pi_R\right)=k=\left(\bigwedge k^{e}\right)_0.\label{eq20250422b}
\end{gather}
\end{prop}

\begin{para}\label{para20240724b}
Throughout this note, we understand that $k$-algebras do not necessarily have unity. For instance, unlike $A_R$ which is a $k$-algebra with unity, both $\HH(K^R(\mathbf{x}, \m_R))$ (from the triangle~\eqref{tri}) and $A_R^+$ are $k$-algebras without unity.
\end{para}


The following lemma plays a crucial role in the proof of Main Theorem.

\begin{lem}\label{thm20240726a}
There is an isomorphism  
$$
\HH(K^R(\mathbf{x}, \m_R)) \cong A_R ^+  \ltimes \shift^{-1}\left(\bigwedge k^{e}\right) ^+
$$
of $k$-algebras (without unity). More precisely, there is an isomorphism
\begin{equation}\label{eq20240726a}
\HH(K^R(\mathbf{x}, \m_R)) \cong A_R ^+  \oplus \shift^{-1}\left(\bigwedge k^{e}\right) ^+ 
\end{equation}
of $k$-vector spaces which is also an isomorphism of $k$-algebras if the multiplication on the right-hand side of~\eqref{eq20240726a} is defined by the formula 
\begin{equation}\label{eq20250422c}
(u, y)\cdot (u', y') = (uu', 0) 
\end{equation}
for  all $(u, y),  (u', y') \in A_R ^+  \oplus \shift^{-1}\left(\bigwedge k^{e}\right) ^+$.
\end{lem}

\begin{proof}
It follows from the triangle~\eqref{tri} that  $\HH(K^R(\mathbf{x}, \m_R)) \cong  \shift \Cone(\pi_R)$ as a graded $k$-module, where $\Cone(\pi_R)$ is the mapping cone of $\pi_R$. 
On the other hand, by Proposition~\ref{cor20250422a} we get an isomorphism $\shift \Cone(\pi_R) \cong A_R ^+  \oplus \shift^{-1}\left(\bigwedge k^{e}\right) ^+$, which implies the isomorphism~\eqref{eq20240726a}.
In this isomorphism, the induced mapping $$\theta\colon \HH(K^R(\mathbf{x}, \m_R)) \to A_R ^+$$ is a natural one which preserves multiplication and hence, is a ring homomorphism. Also, $\shift^{-1}\left(\bigwedge k^{e}\right) ^+=\ker(\theta)$ is identical to the image of the composition
$$
\xymatrix{
\shift^{-1}\left(\bigwedge k^{e}\right) ^+ \ar[r]^-{\rho}&  \shift^{-1}\bigwedge k^{e} \ar[r]^-{\psi}& \HH(K^R(\mathbf{x}, \m_R))
}
$$
where $\rho$ is the natural inclusion map.

To verify~\eqref{eq20250422c}, let $y \in \shift^{-1}\left(\bigwedge k^{e}\right) ^+$ and find its pre-image  $\eta \in \shift^{-1}K^R(\mathbf{x})^+$ that is mapped to $y$ by applying the functor $- \otimes _R k$. 
Then, $\psi (y)$  is defined to be the class of $-d^{K^R(\mathbf{x})}(\eta)$  in $\HH(K^R(\mathbf{x}, \m_R))$. Note that $d^{K^R(\mathbf{x})}(\eta)\in Z(K^R(\mathbf{x}))$.

Let $y'\in \shift^{-1}\left(\bigwedge k^{e}\right) ^+$ and consider an element $\eta'\in \shift^{-1}K^R(\mathbf{x})^+$ that is obtained for $y'$ similar to $\eta$ for $y$ in the previous paragraph. 
Then, the multiplication  $yy'$ in $\HH(K^R(\mathbf{x}, \m_R))$ is, by definition, the product 
$d^{K^R(\mathbf{x})}(\eta)d^{K^R(\mathbf{x})}(\eta')$ in $\HH(K^R(\mathbf{x}, \m_R))$. Note that $\eta d^{K^R(\mathbf{x})}(\eta')\in K^R(\mathbf{x}, \m_R)$ and $d^{K^R(\mathbf{x})}(\eta)d^{K^R(\mathbf{x})}(\eta')$ equals to the boundary $d^{K^R(\mathbf{x})}\left(\eta d^{K^R(\mathbf{x})}(\eta')\right)$  in $K^R(\mathbf{x}, \m_R)$. 
Hence, $yy' =0$  in  $\HH(K^R(\mathbf{x}, \m_R))$. 

Next, let $\zeta  \in Z(K^R(\mathbf{x}))^+$, and denote the class of $\zeta$ in $\HH(K^R(\mathbf{x}))$ by $z$. 
Then, the multiplication  $yz$  in  $\HH(K^R(\mathbf{x}, \m_R))$ is defined to be the class of 
$d^{K^R(\mathbf{x})}(\eta) \zeta$ which is equal to  $d^{K^R(\mathbf{x})}(\eta \zeta)$ because $d^{K^R(\mathbf{x})}(\zeta)=0$. 
Again, by the equality~\eqref{eq20250421a} we know that $\zeta  \in Z(K^R(\mathbf{x}, \m_R))^+ \subseteq  K^R(\mathbf{x}, \m_R)$. Hence, $\eta \zeta \in K^R(\mathbf{x}, \m_R)$ and therefore, $d^{K^R(\mathbf{x})}(\eta) \zeta$ is homotopic to zero in $K^R(\mathbf{x}, \m_R)$. This means that $yz = 0$  in $\HH(K^R(\mathbf{x}, \m_R))$. This completes the verification of the equality~\eqref{eq20250422c}.
\end{proof}

\begin{para}\label{para20250510c}
The \emph{fiber product} of commutative noetherian local rings $(S,\frak s,k)$ and $(T,\frak t,k)$ over their common residue field $k$ is defined to be
$$
S\times_k T=\left\{(s,t)\in S\times T\mid \pi_S(s)=\pi_T(t)\right\}
$$
where $S\xra{\pi_S} k\xla{\pi_T}T$ are the natural surjections. In this situation, $S\times_k T$ is a local ring with maximal ideal $\frak s\oplus \frak t$ and residue field $k$.
Note that the class of fiber product rings of the form $S\times_kT$ coincides with the class of local rings with decomposable maximal ideal. More precisely, if $\fm_R=I\oplus J$ is a non-trivial direct sum decomposition of $\fm_R$ with $I,J$ being ideals of $R$, then $R\cong R/I\times_kR/J$; see either~\cite[Lemma 3.1]{ogoma:edc} or~\cite[Fact 3.1]{nasseh:lrqdmi}. Moreover, for non-negative integers $a,b$ and variables $\mathbf{v}=v_1,\ldots,v_n$ and $\mathbf{w}=w_1,\ldots,w_m$, there is a ring isomorphism
\begin{equation}\label{eq20250511a}
\!\!\!\!\!\frac{k[\![\mathbf{v}]\!]}{(f_1,\ldots,f_a)}\times_k \frac{k[\![\mathbf{w}]\!]}{(g_1,\ldots,g_b)}\cong\frac{k[\![\mathbf{v},\mathbf{w}]\!]}{\left(f_1,\ldots,f_a,g_1,\ldots,g_b,v_iw_j \mid \begin{matrix}\!\!1\leq i\leq n\\ 1\leq j\leq m\end{matrix}\right)}.
\end{equation}
Other references on fiber products (i.e., local rings with decomposable maximal ideal) and their properties include, but not limited to,~\cite{nasseh:vetfp, NST, nasseh:ahplrdmi}.
\end{para}

We can now provide the proof of Main Theorem. Evidently, in addition to the terminology from~\ref{para20250420b}-\ref{para20250420c}, the setting of this theorem is now fully in effect.\vspace{2mm}

\noindent \emph{Proof of Main Theorem.} In the setting of Main Theorem, as we discussed in~\ref{para20250510c}, note that $R\cong R/I\times_k R/J$ is a fiber product ring over the residue field $k$ and we also have $\fm_{R/I}\cong J$ and $\fm_{R/J}\cong I$. Since $\mathbf{x}_2 I=(0)=\mathbf{x}_1 J$ in $R$, we have the isomorphisms 
\begin{align*}
K^R(\mathbf{x}, I)=K^R(\mathbf{x}_1\cup \mathbf{x}_2, I)\cong K^R(\mathbf{x}_1, I)\otimes_R \bigwedge R^{s}\\
K^R(\mathbf{x}, J)=K^R(\mathbf{x}_1\cup \mathbf{x}_2, J)\cong \bigwedge R^{t} \otimes_R K^R(\mathbf{x}_2, J)
\end{align*}
of $R$-algebras. Therefore, we get the isomorphisms
\begin{align*}
\HH(K^R(\mathbf{x}, \m_R)) &\cong \HH\left(K^R(\mathbf{x}, I)\right)\oplus \HH\left(K^R(\mathbf{x}, J)\right)\\
&\cong \left(\HH(K^R(\mathbf{x}_1, I)) \otimes _R \bigwedge R^{s} \right) \oplus 
\left(\bigwedge R^{t}\otimes_R \HH(K^R(\mathbf{x}_2, J))\right)\\
&\cong \left(\HH(K^R(\mathbf{x}_1, I)) \otimes _k \bigwedge k ^{s} \right) \oplus 
\left(\bigwedge k ^{t}\otimes_k \HH(K^R(\mathbf{x}_2, J))\right)\\
&\cong
\left(\HH(K^{R/J}(\mathbf{x}_1, I)) \otimes _k \bigwedge k ^{s} \right) \oplus 
\left(\bigwedge k ^{t}\otimes_k \HH(K^{R/I}(\mathbf{x}_2, J))\right)
\end{align*}
in which the third isomorphism follows from the fact that both $\HH(K^R(\mathbf{x}_1, I))$ and $\HH(K^R(\mathbf{x}_2, J))$ are $k$-algebras. In fact, these are all isomorphisms of $k$-algebras.
Now, applying Lemma~\ref{thm20240726a} to the rings $R$, $R/I$, and $R/J$ we obtain the isomorphisms
\begin{align}
A_R ^+  \ltimes \shift^{-1}\left(\bigwedge k^{e}\right) ^+ &\cong
\left(\left(A^+_{R/J} \ltimes  \shift^{-1}\left(\bigwedge k^{t}\right)^+\right) \otimes_k  \bigwedge k ^{s} \right)\notag\\
&\oplus\left( \bigwedge k ^{t} \otimes_k \left(A^+_{R/I} \ltimes \shift^{-1} \left(\bigwedge k^{s}\right)^+\right)  \right)\notag  \\
&\cong \left( \left(A ^+_{R/J} \otimes _k \bigwedge k^{s}\right) \times \left(\bigwedge k ^{t}\otimes_k A ^+_{R/I}\right)  \right)\ltimes \shift^{-1} \Gamma\label{eq20250620a}
\end{align}
of $k$-algebras, where
$$
\Gamma=\left(\left(\bigwedge k^{t}\right)^+ \otimes_k  \bigwedge k ^{s}\right) \oplus 
\left(\bigwedge k ^{t} \otimes_k \left(\bigwedge k^{s}\right)^+\right).
$$
On the other hand, there is a short exact sequence\footnote{It is elementary to see that for an $R$-module $M$ and its $R$-submodules $M_1$ and $M_2$, there is a short exact sequence
$$
0\to \frac{M}{M_1\cap M_2}\overset{\alpha}\longrightarrow \frac{M}{M_1}\oplus \frac{M}{M_2}\overset{\beta}\longrightarrow\frac{M}{M_1+M_2}\to 0
$$
of $R$-modules in which for all elements $x,y\in M$ the maps $\alpha$ and $\beta$ are defined by the equalities
\begin{gather*}
\alpha(x)=(x+M_1,x+M_2)\\
\beta(x+M_1,y+M_2)=x-y+(M_1+M_2).
\end{gather*} 
}
$$
0\to\frac{\bigwedge k^{e}}{k\otimes_kk}\longrightarrow \frac{\bigwedge k^{e}}{k\otimes_k\bigwedge k^{t}}\oplus \frac{\bigwedge k^{e}}{\bigwedge k^{s}\otimes_kk}\longrightarrow \frac{\bigwedge k^{e}}{k\otimes_k\bigwedge k^{t}+\bigwedge k^{s}\otimes_kk}\to 0
$$
of $k$-vector spaces in which the left term is isomorphic to $\left(\bigwedge k^{e}\right) ^+$, the middle term is isomorphic to $\Gamma$, and the term on the right is isomorphic to $\shift W$.
In other words, $\left(\bigwedge k^{e}\right) ^+$ is naturally embedded as a direct summand in $\Gamma$ and the cokernel of this embedding is $\shift W$. Hence, we have the isomorphism
$\Gamma\cong \left(\bigwedge k^{e}\right) ^+\oplus \shift W$ of $k$-vector spaces which along with~\eqref{eq20250620a} implies the isomorphism~\eqref{eqa}. \qed

\begin{para}\label{para20250510a}
Consider the setting of Main Theorem. The proof of Lemma~\ref{thm20240726a} shows that $A^+_R\cdot W=0$. Moreover, in formula~\eqref{eq20250510a}, identifying $\bigwedge k^{t}\otimes_k\bigwedge k^{s}$ by $\bigwedge k^{t+s}=\bigwedge k^{e}$ and considering a basis $\{\omega_1,\ldots,\omega_t,\omega_{t+1},\ldots,\omega_{e}\}$ for the $k$-vector space $k^{e}$, we see that the $k$-vector space $\shift W$ is generated by the set
$$
\left\{\omega_{i_1}\wedge\ldots\wedge \omega_{i_v}
\left| \text{\begin{tabular}{c}
\!\!\!\!\!\!\!\!$1\leq i_1<\cdots<i_v\leq t+s$\\
$\{i_1,\ldots,i_v\}\nsubseteq \{1,\ldots,t\}$ and \\
\!\!$\{i_1,\ldots,i_v\}\nsubseteq \{t+1,\ldots,e\}$
\end{tabular}}\right.\!\!\!\right\}.
$$
Note that the vector space dimension of $W$ is $2^{t+s}-2^t-2^s+1=(2^t-1)(2^s-1)$.
\end{para}

We continue using the setting of Main Theorem in the following examples.

\begin{ex}\label{ex20250510a}
Let $R=k[\![x]\!]\times_kk[\![y]\!]\cong k[\![x,y]\!]/(xy)$; see~\eqref{eq20250511a}. In this case, $\m_R=(x)\oplus (y)$, $t=1=s$, and $R/(x)$ and $R/(y)$ are regular local rings. Hence, $A^+_{R/(x)}=0=A^+_{R/(y)}$, that is, $A_{R/(x)}=k=A_{R/(y)}$. On the other hand, by~\ref{para20250510a}, the $k$-vector space $\shift W$ is generated by a singleton $\{\omega\wedge \omega'\}$ of degree 2, where $\{\omega\}$ and $\{\omega'\}$ are bases for $k=k^t$ and $k=k^s$, respectively. Therefore, $W\cong \shift^{-1}\left(\shift^2 k\right)=\shift k$, and hence, $A_R\cong k\ltimes \shift k$; compare this formula with case $\mathbf{C}(1)$ in~\cite[1.3]{avramov:cslrec3}.
\end{ex}

\begin{ex}
Let $R=\left(k[\![x]\!]/(x^{m})\right)\times_k\left(k[\![y]\!]/(y^{n})\right)\cong k[\![x,y]\!]/(x^{m},xy,y^{n})$, for integers $m,n\geq 2$; see \eqref{eq20250511a}. In this case, $\m_R=(\overline{x})\oplus (\overline{y})$, where $\overline{x},\overline{y}$ denote $x,y$ modulo $(x^{m})$ and $(y^{n})$. Also, $t=1=s$ and we have $R/(\overline{y})\cong k[\![x]\!]/(x^{m})$ and $R/(\overline{x})\cong k[\![y]\!]/(y^{n})$. Hence, $A^+_{R/(\overline{x})}=k\omega=A^+_{R/(\overline{y})}$, where $\omega$ is a generator of degree $1$; more precisely, by~\cite[1.3]{avramov:cslrec3} we have $A_{R/(\overline{x})}=\bigwedge k\omega=A_{R/(\overline{y})}$. Similar to Example~\ref{ex20250510a}, again as a $k$-vector space, $\shift W$ is generated by a singleton of degree 2 and thus, by Main Theorem we have
$$
A^+_R\cong \left(\left(\bigwedge k^1\otimes_kk\omega\right)\times \left(k\omega\otimes_k\bigwedge k^1\right)\right)\ltimes \shift k.
$$
Note that, as a $k$-vector space, $A_R\cong k\oplus \shift k^3\oplus\shift^2 k^2$. Note also that, using the ring structure of $A_R^+$ one can check that $\left(A_R^+\right)^2=0$. Therefore, we have an isomorphism $A_R\cong k\ltimes A_R^+$ of $k$-algebras; compare this with case $\mathbf{S}$ in~\cite[1.3]{avramov:cslrec3}. 
\end{ex}


Our Main Theorem can be generalized to the following statement by induction.

\begin{thm}\label{rproducts}
Assume that there are non-zero ideals $I_i$ with $i=1,\ldots,r$ for an integer $r\geq 2$ such that $\fm_R=\bigoplus_{1\leq i\leq r}I_i$, where each $I_i$ is minimally generated by $n_i$ elements. For each $1\leq i\leq r$ we set $J_i=I_1\oplus\cdots\oplus I_{i-1}\oplus I_{i+1}\oplus\cdots\oplus I_r$. Then, there is an isomorphism
$$A _R ^+ \cong \mathcal{A} \ltimes W$$ of $k$-algebras (without unity) in which
\begin{gather*}
\mathcal{A} = \prod _{i=1} ^r \left(\bigwedge k^{n_1}\otimes_k \cdots \otimes_k \overbrace{A _{R/J_i} ^+}^{\substack{i-\text{th place}}} \otimes_k \cdots  \otimes_k \bigwedge k^{n_r} \right)
\\
W=\shift^{-1}\left(\frac{\bigwedge k^{n_1}\otimes_k \ldots \otimes_k \bigwedge k^{n_r}}{\sum _{i=1}^r (\bigwedge k^{n_1}\otimes_k \cdots \otimes_k \underbrace{k}_{i-\text{th place}} \otimes_k \cdots \otimes_k \bigwedge k^{n_r})}\right)
\end{gather*}
and the $\mathcal{A}$-module structure on  $W$ is given by the equality $\mathcal{A}\cdot W=0$.
\end{thm}

Using Theorem~\ref{rproducts}, we can extend Example~\ref{ex20250510a} as follows.

\begin{ex}\label{ex20250509a}
Let $r\geq 2$ be an integer and $R=T_1\times_k \cdots\times_k T_r$, where for each $1\leq i\leq r$ we have $T_i=k[\![\chi_i]\!]$ with $\chi_i=\{X_1^{(i)},\ldots,X_{n_i}^{(i)}\}$ for an integer $n_i\geq 1$. In this example, by~\eqref{eq20250511a}, we have the isomorphism $R\cong k[\![\chi_1,\ldots,\chi_r]\!]/\frak I$, where $\frak I$ is the ideal of the ring $k[\![\chi_1,\ldots,\chi_r]\!]$ generated by the set $\left\{\chi_i\chi_j\mid 1\leq i,j\leq r\right\}$. Note that, $\m_R= (\chi_1)\oplus\cdots\oplus(\chi_r)$ and if $J_i=(\chi_1)\oplus\cdots\oplus(\chi_{i-1})\oplus (\chi_{i+1})\oplus \cdots\oplus(\chi_r)$ for each $1\leq i\leq r$, then
$$
R/J_i\cong T_i
$$
which is a regular local ring. Thus, for all $1\leq i\leq r$ we have $A^+_{R/J_i}=0$, which implies that $A_R\cong (k\oplus\mathcal{A})\ltimes W \cong k\ltimes W$. To compute $W$, for every integer $2\leq p\leq r$, let $\ell_p$ denote the number of degree-$p$ basis elements in the $k$-vector space $\shift W$, i.e., $W_{p-1}=\shift^{p-1} k^{\ell_p}$. According to our discussion in~\ref{para20250510a}, we have
$$
\ell_p=\sum_{j_1+\cdots+j_r=p} {n_1\choose j_1}\cdots{n_r\choose j_r}
$$
where each $j_i$ is a positive integer. This completely determines the structure of $A_R$.
\end{ex}

The following example demonstrates Example~\ref{ex20250509a} for $r=n_1=n_2=2$.

\begin{ex}
Let $R=k[\![x,y]\!]\times_kk[\![z,w]\!]\cong k[\![x,y,z,w]\!]/(xz,xw,yz,yw)$. In this case, $\m_R=(x,y)\oplus (z,w)$, $n_1=2=n_2$, and we have $R/(x,y)\cong k[\![z,w]\!]$ and $R/(z,w)\cong k[\![x,y]\!]$ are regular local rings. Hence, $A^+_{R/(x,y)}=0=A^+_{R/(z,w)}$, i.e., $A_R\cong k\ltimes W$. By~\ref{para20250510a}, the $k$-vector space $\shift W$ is generated by the set\vspace{1mm}
$$
\!\!\!\left\{\!\omega_{i_1}\wedge\ldots\wedge \omega_{i_v}\!
\left| \text{\begin{tabular}{c}
\!\!$1\leq i_1<\cdots<i_v\leq 4$\\
\!\!\!\!\!\!$\{i_1,\ldots,i_v\}\nsubseteq \{1,2\}$ \\
\!\!\!\!\!\!$\{i_1,\ldots,i_v\}\nsubseteq \{3,4\}$
\end{tabular}}\right.\!\!\!\!\right\}\!=\!\left\{
\text{\begin{tabular}{c}
\!\!\!\!\!\!$\omega_1\wedge \omega_3, \omega_1\wedge \omega_4, \omega_2\wedge \omega_3, \omega_2\wedge \omega_4,$\\
\!\!\!\!$\omega_1\wedge \omega_2\wedge \omega_3, \omega_1\wedge \omega_2\wedge \omega_4, \omega_1\wedge \omega_3\wedge \omega_4,$\\
\!\!\!\!\!\!$\omega_2\wedge \omega_3\wedge \omega_4,\omega_1\wedge \omega_2\wedge \omega_3\wedge \omega_4$
\end{tabular}}
\!\!\!\!\right\}\vspace{1mm}
$$
and therefore, $\shift W\cong \shift^2 k^4 \oplus \shift^3 k^4\oplus \shift^4 k$. Hence, $A_R\cong k\ltimes \left(\shift k^4 \oplus \shift^2 k^4\oplus \shift^3 k\right)$; again, compare with case $\mathbf{H}(0,0)$ in~\cite[1.3]{avramov:cslrec3}. Considering the notation from Example~\ref{ex20250509a}, note that $\ell_2=4=\ell_3$ and $\ell_4=1$.
\end{ex}

\begin{para}\label{para20250502a}
Tate~\cite{Tate} proved that there is a DG algebra resolution $T(R)\xra{\simeq} \widehat{R}$ over $Q$ in which $T(R)=Q\langle X_i\mid 1\leq i\leq n\rangle$, with $n\leq \infty$, is a free DG algebra extension of $Q$ obtained by adjoining finitely or infinitely countably many variables $X_i$ to $Q$; for the notation, see~\cite{NOY1}. It follows from~\eqref{eq20240713a} and~\eqref{eq20250502a} that
\begin{equation}\label{eq20250502b}
A_R\cong \HH(k\otimes_Q T(R)).
\end{equation}
In the setting of Main Theorem, we know that $A_R$, i.e., the left-hand side of~\eqref{eq20250502b}, can be completely determined if $A_{R/I}$ and $A_{R/J}$ are given. However, we do not know how, or whether, the Tate resolutions $T(R)$, $T(R/I)$, and $T(R/J)$ are related as well. The only result in this direction, that we are aware of and partially discusses this under restrictions, is obtained by Geller~\cite[Theorem 3.4]{Geller}. Thus, we conclude this paper with the following general question. 
\end{para}

\begin{question}
Under the settings of Main Theorem and~\ref{para20250502a}, can one describe the Tate resolution $T(R)$ in terms of $T(R/I)$ and $T(R/J)$? 
\end{question}



\section*{Acknowledgments}
We are grateful to Srikanth Iyengar, Keri Ann Sather-Wagstaff, and Keller VandeBogert for their comments on an earlier version of this paper.

\providecommand{\bysame}{\leavevmode\hbox to3em{\hrulefill}\thinspace}
\providecommand{\MR}{\relax\ifhmode\unskip\space\fi MR }
\providecommand{\MRhref}[2]{%
  \href{http://www.ams.org/mathscinet-getitem?mr=#1}{#2}
}
\providecommand{\href}[2]{#2}

\end{document}